\documentclass[10pt,a4paper]{article}
\usepackage{amsmath}
\usepackage{amsfonts}
\usepackage{amsthm}
\usepackage{relsize}

\newtheorem{thm}{Theorem}[section]
\newtheorem{lem}[thm]{Lemma}
\newtheorem{prop}[thm]{Proposition}
\newtheorem{cor}[thm]{Corollary}

\newtheorem{que}{Question}
\newtheorem*{propn}{Proposition}

\newtheorem*{thma}{Theorem A}
\newtheorem*{thmb}{Theorem B}
\newtheorem*{thmc}{Theorem C}
\newtheorem*{thmd}{Theorem D}
\newtheorem*{thme}{Theorem E}

\newcommand{\Aut}{\mathrm{Aut}}
\newcommand{\Inn}{\mathrm{Inn}}

\newcommand{\Core}{\mathrm{Core}}

\newcommand{\Ob}{\mathrm{Ob}}
\newcommand{\ob}{\mathrm{ob}}
\newcommand{\St}{\mathrm{St}}
\newcommand{\Fin}{\mathrm{Fin}}

\newcommand{\Comm}{\mathrm{Comm}}
\newcommand{\OI}{\mathrm{OI}}
\newcommand{\Ilhd}{\mathrm{I}^\lhd}

\newcommand{\bN}{\mathbb{N}}

\newcommand{\bQ}{\mathbb{Q}}
\newcommand{\bZ}{\mathbb{Z}}

\newcommand{\mcA}{\mathcal{A}}
\newcommand{\mcB}{\mathcal{B}}
\newcommand{\mcC}{\mathcal{C}}

\newcommand{\mcF}{\mathcal{F}}

\newcommand{\mcH}{\mathcal{H}}
\newcommand{\mcI}{\mathcal{I}}
\newcommand{\mcJ}{\mathcal{J}}
\newcommand{\mcK}{\mathcal{K}}

\newcommand{\mcN}{\mathcal{N}}
\newcommand{\mcO}{\mathcal{O}}

\newcommand{\mcX}{\mathcal{X}}

\begin{document}

\title{On the structure of just infinite profinite groups}

\author{Colin Reid\\
School of Mathematical Sciences\\
Queen Mary, University of London\\
Mile End Road, London E1 4NS\\
c.reid@qmul.ac.uk}

\maketitle

\begin{abstract}A profinite group $G$ is just infinite if every closed normal subgroup of $G$ is of finite index.  We prove that an infinite profinite group is just infinite if and only if, for every open subgroup $H$ of $G$, there are only finitely many open normal subgroups of $G$ not contained in $H$.  This extends a result recently established by Barnea, Gavioli, Jaikin-Zapirain, Monti and Scoppola in \cite{BGJMS}, who proved the same characterisation in the case of pro-$p$ groups.  We also use this result to establish a number of features of the general structure of profinite groups with regard to the just infinite property.\end{abstract}


\paragraph{NB} Throughout this paper, we will be concerned with profinite groups as topological groups.  As such, it will be tacitly assumed that subgroups are required to be closed, that homomorphisms are required to be continuous, and that generation refers to topological generation.  When we wish to suppress topological considerations, the word `abstract' will be used, for instance `abstract subgroup'.

\section{Introduction}

\paragraph{Definitions}A profinite group $G$ is \emph{just infinite} if it is infinite, and every non-trivial normal subgroup of $G$ is of finite index.  Say $G$ is \emph{hereditarily just infinite} if every open subgroup of $G$ is just infinite, including $G$ itself.\\

The main theorem of this paper is the following:

\begin{thma}[Generalised obliquity theorem]Let $G$ be an infinite profinite group.  Then the following are equivalent:

(i) $G$ is just infinite;

(ii) The set $\mcK_H = \{K \unlhd_o G \mid K \not\le H\}$ is finite for every open subgroup $H$ of $G$;

(iii) there exists a family $\mcF$ of open subgroups of $G$ with trivial intersection, such that $\mcK_H$ is a finite set for every $H \in \mcF$.\end{thma}

This generalises Theorem 36 of \cite{BGJMS}, in which the above theorem is proved under the assumption that $G$ is a pro-$p$ group.\\

It follows from the theorem that just infinite profinite groups have finitely many open subgroups of given index.  This is trivial in the case of just infinite virtually pro-$p$ groups, since they are always finitely generated; but a just infinite profinite group need not be finitely generated in general.\\

To prove results about just infinite profinite groups, we will make use of normal Frattini subgroups, as defined below.  Similar methods are used by Zalesskii in \cite{Zal} to give conditions under which a profinite group has a just infinite image, but the idea is also useful for studying just infinite profinite groups themselves.

\paragraph{Definitions}The \emph{normal Frattini subgroup} $\Phi^\lhd(G)$ of a profinite group $G$ is the intersection of all maximal closed normal subgroups, or equivalently the intersection of the open normal subgroups $N$ such that $G/N$ is simple.  Note that if $G$ is non-trivial then $\Phi^\lhd(G)$ is necessarily a proper subgroup.  We say $G$ is a \emph{$\Phi^\lhd$-group} if $\Phi^\lhd(G)$ has finite index in $G$; say $G$ is a \emph{hereditary $\Phi^\lhd$-group} if every open subgroup of $G$ is a $\Phi^\lhd$-group.\\

In this paper we will give a description of profinite groups $G$ for which $\Phi^\lhd(G)=1$ (Lemma \ref{philhdlem} (ii)), from which it will follow that every just infinite profinite group is a hereditary $\Phi^\lhd$-group.

\paragraph{Definitions}Given a profinite group $G$, define $\Phi^{\lhd n}(G)$ by $\Phi^{\lhd 0}(G) = G$ and thereafter $\Phi^{\lhd (n+1)}(G) = \Phi^\lhd(\Phi^{\lhd n}(G))$.  Define the \emph{$\Phi^\lhd$-height} of a finite group to be the least $n$ such that $\Phi^{\lhd n}(G)=1$.

Let $\mcX$ be a class of finite groups.  The \emph{$\mcX$-residual} $O^\mcX(G)$ of a profinite group $G$ is the intersection of all open normal subgroups $N$ such that $G/N \in \mcX$.  Say $G$ is \emph{residually-$\mcX$} if $O^\mcX(G)=1$.  Given a hereditary $\Phi^\lhd$-group $G$, note that $\Phi^{\lhd n}(G)=O^\mcX(G)$, where $\mcX$ is the class of finite groups of $\Phi^\lhd$-height at most $n$, and that $O^\mcX(G)$ is therefore an open subgroup of $G$.  The \emph{$\mcX$-radical} $O_\mcX(G)$ is the subgroup generated by all subnormal $\mcX$-subgroups.  A \emph{radical} of $G$ is a subgroup which is the $\mcX$-radical of $G$ for some class $\mcX$.\\

Given a profinite group $G$ and subgroup $H$, we define the \emph{oblique core} $\Ob_G(H)$ and \emph{strong oblique core} $\Ob^*_G(H)$ of $H$ in $G$ as follows:
\[ \Ob_G(H) := H \cap \bigcap \{K \unlhd_o G \mid K \not\le H\}\]
\[ \Ob^*_G(H) := H \cap \bigcap \{K \leq_o G \mid H \le N_G(K), \; K \not\le H\}\]
Note that $\Ob_G(H)$ and $\Ob^*_G(H)$ have finite index in $H$ if and only if the relevant intersections are finite.\\

As motivation for the term `generalised obliquity', we recall the following definition:

\paragraph{Definition}[Klaas, Leedham-Green, Plesken \cite{KLP}] Let $G$ be a pro-$p$ group for which each lower central subgroup is an open subgroup. Then the \emph{$i$-th obliquity} of $G$ is given as follows (with the obvious convention that $\log_p(\infty) = \infty$):
\[ o_i(G) := \log_p (|\gamma_{i+1}(G):\Ob_G(\gamma_{i+1}(G))|).\]
The \emph{obliquity} of $G$ is given by $o(G) := \sup_{i \in \bN}o_i(G)$.\\

It is an immediate consequence of Theorem A that if $G$ is a pro-$p$ group such that each lower central subgroup is an open subgroup, then $G$ is just infinite if and only if $o_i(G)$ is finite for every $i$.\\

Theorem A also gives a characterisation of the hereditarily just infinite property: an infinite profinite group $G$ is hereditarily just infinite if and only if $\Ob^*_G(H)$ has finite index for every open subgroup $H$ of $G$.

\paragraph{Definitions}Given a profinite group $G$, let $\Ilhd_n(G)$ denote the intersection of all open normal subgroups of $G$ of index at most $n$.  Now define $\OI_n(G)$ to be $\Ob_G(\Ilhd_n(G))$, and $\OI^*_n(G)$ to be $\Ob^*_G(\Ilhd_n(G))$.  We thus obtain functions $\ob_G$ and $\ob^*_G$ from $\bN$ to $\bN \cup \{ \infty \}$ defined by
\[ \ob_G(n) := |G:\OI_n(G)| \quad ; \quad \ob^*_G(n) := |G:\OI^*_n(G)|. \]
These are respectively the \emph{generalised obliquity function} or \emph{$\ob$-function} and the \emph{strong generalised obliquity function} or \emph{$\ob^*$-function} of $G$.  Given a function $\eta$ from $\bN$ to $\bN$, let $\mcO_\eta$ denote the class of profinite groups for which $\ob_G(n) \leq \eta(n)$ for every $n \in \bN$, and let $\mcO^*_\eta$ denote the class of profinite groups for which $\ob^*_G(n) \leq \eta(n)$ for every $n \in \bN$.\\

These functions give characterisations of the just infinite property and the hereditarily just infinite property in terms of finite images, as described in the following:

\begin{thmb}Let $G$ be a profinite group.  Let $\mcC_\eta$ indicate either $\mcO_\eta$ or $\mcO^*_\eta$.  Then the following are equivalent:

(i) $G$ is finite or a $\mcJ$-group, where $\mcJ$ is the class of just infinite groups if $\mcC_\eta = \mcO_\eta$, or the class of hereditarily just infinite groups if $\mcC_\eta = \mcO^*_\eta$;

(ii) There is some $\eta$ for which $G$ is a $\mcC_\eta$-group;

(iii) There is some $\eta$ for which every image of $G$ is a $\mcC_\eta$-group;

(iv) There is some $\eta$, and some family of normal subgroups $\{N_i \mid i \in I\}$, such that each $G/N_i$ is a $\mcC_\eta$-group and such that $G \cong \varprojlim G/N_i$.\\

Moreover, (ii), (iii) and (iv) are equivalent for any specified $\eta$.\end{thmb}

\paragraph{Remark}In more specific contexts, the subgroups $\Ilhd_n(G)$ in the definition of (strong) generalised obliquity functions can be replaced with various other characteristic open series, and Theorem B would remain valid, with essentially the same proof.  For instance, in case of pro-$p$ groups, one could use the lower central exponent-$p$ series, and in the case of prosoluble groups with no infinite soluble images, one could use the derived series.\\

The definitions of $\ob$-functions and $\ob^*$-functions lead to the following general question:

\begin{que}Which functions from $\bN$ to $\bN$ can occur as $\ob$-functions or $\ob^*$-functions for (hereditarily) just infinite profinite groups?  What growth rates are possible?\end{que}

As a straightforward example, consider the significance of linear growth of the $\ob$-function, given a pro-$p$ group $G$.

\begin{propn}Let $G$ be a pro-$p$ group.  The following are equivalent:

(i) $G$ is either $\bZ_p$ or has finite obliquity;
 
(ii) there is a constant $k$ such that $\ob_G(n) \leq kn$ for all $n$.\end{propn}

Another example is that of self-reproducing branch groups, in the sense of \cite{GriNH}; in particular, such a group $G$ has an open subgroup that is isomorphic to a direct product of copies of $G$.  Here, self-similarity properties can be used to obtain a bound on the growth rate of the obliquity function.

\begin{propn}Let $G$ be a just infinite profinite branch group that is self-reproducing at some vertex (see later).  Then there is a constant $c$ such that $\ob_G(n) \leq c^n$ for all $n$.\end{propn}

We also consider how the $\ob$-function and $\ob^*$-function of a just infinite profinite group $G$ relate to those of its open normal subgroups.

\begin{thmc}Let $G$ be a just infinite profinite group, and let $H$ be a subgroup of $G$ of index $h$.
  
(i) The following inequality holds for sufficiently large $n$:
\[ h^{-1} \ob_G(hn) \leq \ob_H(n). \]

(ii) Let $t=|G:\Core_G(H)|$. The following inequality holds for all $n$:
\[\ob^*_H(n) \leq h^{-1}\ob^*_G(tn^h).\]

(iii) For a given $n$, let $\mcI_n$ be the set of subgroups of $G$ containing $\Ilhd_n(G)$.  The following inequality holds:
\[\ob^*_G(n) \leq \prod_{L \in \mcI_n} |G:L|\ob_L(n).\]\end{thmc}

Another consequence of generalised obliquity concerns non-abelian normal sections of a just infinite profinite group $G$.  In sharp contrast to the case of abelian normal sections, a given isomorphism type of non-abelian finite group can occur only finitely many times as a normal section of $G$.  In fact, more can be said here.

\begin{thmd}Let $G$ be a just infinite profinite group.\\

(i) Let $\mcF$ be a class of non-abelian finite groups, and let $\mcA$ be the class of groups $A$ satisfying $\Inn(F) \leq A \leq \Aut(F)$ for some $F \in \mcF$.  Suppose there are infinitely many pairs $(M,N)$ of normal subgroups of $G$ such that $N \leq M$ and $M/N \in \mcF$.  Then either $G$ is residually-$\mcA$, or it has an open normal subgroup that is residually-$\mcF$.  In particular, at least one of $\mcA$ and $\mcF$ contains groups of arbitrarily large $\Phi^\lhd$-height, and hence $\mcF$ must contain infinitely many isomorphism classes.\\

(ii) Suppose $G$ is not virtually abelian, and let $H$ be a proper open normal subgroup of $G$.  Then $G$ has only finitely many normal subgroups $K$ such that $K/\Phi^\lhd(K') \cong H/\Phi^\lhd(H')$.  In particular, $G$ has only finitely many normal subgroups that are isomorphic to $H$.\end{thmd}

\paragraph{Definition}We say a profinite group $G$ is \emph{index-unstable} if it has a pair of isomorphic open subgroups of different indices, and \emph{index-stable} otherwise.\\

Any just infinite virtually abelian profinite group $G$ is virtually a free abelian pro-$p$ group for some $p$, and hence $G$ is index-unstable.  On the other hand, given part (ii) of Theorem D it seems to be difficult to construct just infinite profinite groups which are index-unstable but not virtually abelian.  We consider the following question:

\begin{que}\label{indexq}Let $G$ be a (hereditarily) just infinite profinite group which is index-unstable.  Is $G$ necessarily virtually abelian?\end{que}

This question is also motivated by the study of commensurators of profinite groups, in the sense of Barnea, Ershov and Weigel (\cite{BEW}).  The relevant definitions will be recalled briefly later, but the reader should consult \cite{BEW} for a detailed account.  Although Question \ref{indexq} remains open, we do obtain some results in this direction.  In particular, the following will be shown:

\begin{thme}Let $G$ be a just infinite profinite group.\\

(i) Suppose there is a proper open subgroup $H$ of $G$ isomorphic to $G$ itself.  Then $G$ is virtually abelian.\\

(ii) Suppose $G$ has infinitely many distinct radicals.  Then $G$ is index-stable.\end{thme}

\paragraph{Remarks}If $G$ is just infinite profinite group that is virtually abelian, it can easily be shown that $G$ is isomorphic to a proper open subgroup of itself if and only if $G$ is a split extension of a free abelian pro-$p$ group for some $p$.\\

Part (ii) of Theorem E is not vacuous, as there are certainly just infinite profinite groups that have infinitely many distinct radicals.  For instance, Ershov (\cite{Ers}) has proved that for $p > 3$, the Nottingham group $J_p$ over the field of $p$ elements satisfies $\Comm(J_p) \cong \Aut(J_p)$; since $J_p$ is hereditarily just infinite, it follows that every characteristic subgroup of $J_p$ is a radical.

\section{The generalised obliquity theorem and some consequences}

\paragraph{Definitions}The \emph{finite radical} $\Fin(G)$ of a profinite group $G$ is the union of all finite normal subgroups of $G$.  This is an abstract subgroup of $G$, though it need not be closed.

\begin{lem}[Wilson, Corollary 3.8 (a) of \cite{WilNH}]Let $G$ be a just infinite profinite group.  Then $G$ has no non-trivial finite subnormal subgroups.\end{lem}

\begin{lem}\label{philhdlem}(i) Let $G$ be a profinite group, and let $H \unlhd G$.  Then $\Phi^\lhd(H) \leq \Phi^\lhd(G)$.\\

(ii) Let $G$ be a profinite group such that $\Phi^\lhd(G)$ is trivial.  Then $G$ is a Cartesian product of elementary abelian groups and non-abelian finite simple groups.  In particular, $\Fin(G)$ is dense in $G$.\end{lem}

\begin{proof}(i)It suffices to show that $\Phi^\lhd(H)$ is contained in every normal subgroup $N$ of $G$ such that $G/N$ is simple; let $N$ be such a subgroup.  Then $HN/N$ is a normal subgroup of $G/N$, so either $HN/N = G/N$ or $HN/N = 1$.  In the former case, $H \cap N$ is a normal subgroup of $H$ such that $H/(H \cap N) \cong HN/N$ is simple, so that $\Phi^\lhd(H) \leq H \cap N$; in the latter case, $H \leq N$.\\

(ii) Let $\mcA$ be the set of open normal subgroups $N$ such that $G/N$ is non-abelian simple, and let $A=\bigcap \mcA$.  Let $\mcB$ be the set of open normal subgroups $N$ such that $G/N$ is cyclic of prime order, and let $B=\bigcap \mcB$.  Then $A \cap B = \Phi^\lhd(G)=1$.  Also $G=BN$ whenever $N \in \mcA$, as $G/BN$ is an image of both an abelian group $G/B$ and a perfect group $G/N$.  Hence $G=AB$ by compactness, and so $G \cong A \times B$.\\

It follows that $A \cong G/B$ is abelian, and hence is a Cartesian product of its Sylow subgroups.  Every finite image of $G/B$ has squarefree exponent, so its Sylow subgroups are all elementary abelian.\\

It also follows that $B \cong G/A$, and so every finite image of $B$ is a subdirect product of non-abelian simple groups; but every finite group that is a subdirect product of non-abelian simple groups is isomorphic to a direct product of non-abelian simple groups.  Hence every finite image of $B$ is isomorphic to a direct product of non-abelian simple groups.  By a standard inverse limit argument, it follows that $B$ is isomorphic to a Cartesian product of non-abelian simple groups.\end{proof}

\begin{cor}\label{philhdcor}(i) Let $G$ be a profinite group, such that every open normal subgroup of $G$ is a $\Phi^\lhd$-group.  Then every open subgroup of $G$ is a $\Phi^\lhd$-group.\\

(ii) Let $G$ be a just infinite profinite group.  Then $G$ is a hereditary $\Phi^\lhd$-group.\end{cor}

\begin{proof}(i) Let $H$ be an open subgroup of $G$, and let $K$ be the core of $H$ in $G$.  Then $K$ is an open normal subgroup of $G$, so $\Phi^\lhd(K)$ has finite index in $K$ and hence in $G$.  Now $\Phi^\lhd(K) \leq \Phi^\lhd(H)$ by the lemma, so $\Phi^\lhd(H)$ has finite index in $H$.\\

(ii) By part (i) it suffices to consider an open normal subgroup $H$ of $G$.  This ensures $\Fin(H)=1$, so $\Phi^\lhd(H)>1$ by the lemma.  This means that $\Phi^\lhd(H)$ is of finite index in $H$, since it is characteristic in $H$ and hence normal in $G$.\end{proof}

Say a set $\mcN$ of open normal subgroups of a profinite group $G$ is \emph{upward-closed} if, given any open normal subgroups $N_1$ and $N_2$ of $G$ such that $N_1 \in \mcN$ and $N_1 \leq N_2$, then $N_2 \in \mcN$.

\begin{lem}\label{phichain}Let $G$ be a hereditary $\Phi^\lhd$-group, and let $\mcK$ be an infinite upward-closed set of normal subgroups of $G$.  Then there is an infinite descending chain $K_1 > K_2 > \dots$ of open normal subgroups of $G$, such that $K_i \in \mcK$ for all $i$.\end{lem}

\begin{proof}Every closed normal subgroup of infinite index is contained in infinitely many open normal subgroups, and so we may assume $\mcK$ consists of open normal subgroups.  Define a directed graph $\Gamma$ with vertex set $\mcK$ as follows: place an arrow from $K_1$ to $K_2$ if $K_2 < K_1$, and there is no $K_3 \in \mcK$ such that $K_2 < K_3 < K_1$.\\

Let $K \in \mcK$.  If there is an arrow from $K$ to another vertex $K_1$, then $K/K_1$ is characteristic-simple by the maximality property of $K_1$ in $K$, so $K_1$ contains $\Phi^\lhd(K)$.  So given $K$, there are finitely many possibilities for $K_1$, corresponding to some of the sections of the finite group $K/\Phi^\lhd(K)$.  Hence each vertex of $\Gamma$ has finite outdegree, and clearly any vertex can be reached from the vertex $G$, so $\Gamma$ has an infinite directed path by K\H{o}nig's lemma; this gives the required descending chain.\end{proof}

\begin{lem}\label{obchain}Let $G$ be a compact topological group, and let $O$ be an open neighbourhood of $1$ in $G$.  Let $K_1 > K_2 > \dots$ be a descending chain of closed normal subgroups of $G$ such that $K_i \not\subseteq O$ for every $i \in I$.  Let $K$ be the intersection of the $K_i$.  Then $K \not\subseteq O$; in particular, $K$ is non-trivial.\end{lem}

\begin{proof}Let $C_i = K_i \cap (G \setminus O)$.  Then each $C_i$ is closed and non-empty, and hence the intersection of finitely many $C_i$ is non-empty, since the $C_i$ form a descending chain.  Since $G$ is compact, it follows that the intersection $K \cap (G \setminus O)$ of all the $C_i$ is non-empty.  Hence $K \not\subseteq O$.\end{proof}

\begin{proof}[Proof of Theorem A]Suppose $\mcK_H$ is infinite for some $H \leq_o G$.  Then $\mcK_H$ is upward-closed, so by Lemma \ref{phichain} there is an infinite descending chain $K_1 > K_2 > \dots$ of open normal subgroups occurring in $\mcK_H$ for which $K_i \not\le H$.  Now apply Lemma \ref{obchain}, to conclude that the intersection $K$ of these $K_i$ is a non-trivial normal subgroup of infinite index.  Hence $G$ is not just infinite, demonstrating that (i) implies (ii).\\

Clearly (ii) implies (iii), so it now suffices to show (iii) implies (i).  Assume (iii), and let $K$ be a non-trivial closed normal subgroup of $G$.  Then there is an element $H$ of $\mcF$ which does not contain $K$.  It follows that $K$, being the intersection of the open normal subgroups of $G$ containing $K$, is the intersection of some open normal subgroups not contained in $H$.  All such subgroups contain $\Ob_G(H)$, which is of finite index, since it is the intersection of a finite set of open normal subgroups of $G$.  Hence $K \geq \Ob_G(H)$, and hence $K$ is open in $G$, proving (i).\end{proof}

\begin{cor}Let $G$ be a just infinite profinite group.\\
  
(i) Let $n$ be an integer.  Then $G$ has finitely many open subgroups of index $n$.\\

(ii) Let $H$ be an infinite profinite group such that every finite image of $H$ is isomorphic to some image of $G$.  Then $G \cong H$.\end{cor}

\begin{proof}(i) Since a subgroup of index $n$ has a core of index at most $n!$, and a normal subgroup of finite index can only be contained in finitely many subgroups, it suffices to consider normal subgroups.  Suppose $G$ has an open normal subgroup $K$ of index $n$.  Then $K$ does not contain any open normal subgroup of $G$ index $n$, other than itself.  Hence by the theorem, the set of such subgroups is finite.\\

(ii) By part (i), $G$ has finitely many open normal subgroups of any given index.  It is shown in \cite{Wil} that in this situation, given any profinite group $H$ such that every finite image of $H$ is isomorphic to an image of $G$, then $H$ is isomorphic to an image of $G$.  Hence there is some $N \lhd G$ such that $G/N \cong H$; since $H$ is infinite and $G$ is just infinite, $N=1$.\end{proof}

\begin{cor}\label{hjicor}Let $G$ be an infinite profinite group.  Then $G$ is hereditarily just infinite if and only if $\Ob^*_G(H)$ has finite index for every open subgroup $H$ of $G$.\end{cor}

\begin{proof}Suppose $G$ has an open subgroup $H$ which is not just infinite.  Then by the theorem, there is an open subgroup $R$ of $H$ which fails to contain infinitely many normal subgroups of $H$, and so $\Ob^*_G(R)$ has infinite index.\\

Conversely, suppose $G$ is hereditarily just infinite.  Let $H$ be an open subgroup of $G$, and let $\mcH$ be the set of subgroups of $G$ containing $H$; then $\mcH$ is finite.  Let $K$ be a subgroup of $G$ such that $H \leq N_G(K)$ but $K \not\le H$.  Let $L= HK$.  Then $L$ is just infinite and $K$ is a normal subgroup of $L$ not containing $H$.  Hence $\Ob^*_G(H)$ contains $\bigcap_{M \in \mcH} \Ob_M(H)$; by the theorem each $\Ob_M(H)$ has finite index, and hence $\Ob^*_G(H)$ has finite index.\end{proof}

\section{A quantitative description of the just infinite property}

\begin{lem}Let $G$ be a profinite group, and let $n$ be a positive integer.\\

(i) Let $N$ be a normal subgroup of $G$.  Then:
\[\Ilhd_n(G)N/N \leq \Ilhd_n(G/N);\]
\[\OI_n(G)N/N \leq \OI_n(G/N);\]
\[\OI^*_n(G)N/N \leq \OI^*_n(G/N).\]

(ii) Let $I$ be a directed set, and let $\mcN = \{N_i \mid i \in I \}$ be a family of normal subgroups of $G$ with trivial intersection, such that $N_i < N_j$ whenever $i > j$.  Let $\pi_i$ be the quotient map from $G$ to $G/N_i$.  Then:
\[ \Ilhd_n(G) = \bigcap_{i \in I} \pi^{-1}_i(\Ilhd_n(G/N_i));\]
\[ \OI_n(G) = \bigcap_{i \in I} \pi^{-1}_i(\OI_n(G/N_i));\]
\[ \OI^*_n(G) = \bigcap_{i \in I} \pi^{-1}_i(\OI^*_n(G/N_i)).\]
\end{lem}

\begin{proof}Let $L = \Ilhd_n(G)$, let $M/N = \Ilhd_n(G/N)$, and let $M_i/N_i = \Ilhd_n(G/N_i)$.\\

(i) If $H/N$ is a normal subgroup of index at most $n$ in $G/N$, then $H$ also has index at most $n$ in $G$.  This proves the first inequality, in other words $L \leq M$.

If $H/N$ is a normal subgroup of $G/N$ not contained in $M/N$, then $H$ is also not contained in $M$ and hence not in $L$.  This proves the second inequality.

If $H/N$ is a subgroup of $G/N$ that is normalised by $M/N$ but not contained in it, then $H$ is also normalised by but not contained in $M$, and hence also normalised by but not contained in $L$.  This proves the third inequality.\\

(ii) Given part (i), it suffices to show for each equation that the left-hand side contains the right-hand side.

If $H$ is a normal subgroup of $G$ index at most $n$, then there is some $N_i$ contained in $H$, which means that $M_i$ is contained in $H$, since $H/N_i$ has index at most $n$ in $G/N_i$.  This proves the first equation, in other words $L = \bigcap_{i \in I} M_i$.

If $H$ is a normal subgroup of $G$ not contained in $L$, then there is some $M_i$ that does not contain $H$, by the first equation.  This proves the second equation.

Let $H$ be an open subgroup of $G$ that is normalised by $L$ but not contained in it.  Then $HL$ is an open subgroup of $G$ which contains $\bigcap_{i \in I} M_i$.  By Lemma \ref{obchain}, this means that there is some $M_i$ contained in $HL$, which implies that this $M_i$ normalises $H$.  By the first equation, there is some $M_j$ not containing $H$.  Now take $M_k \leq M_i \cap M_j$, and note that $\Ob^*_{G/N_k}(M_k/N_k)$ is contained in $H$.  This proves the third equation.\end{proof}

\begin{proof}[Proof of Theorem B]We give the proof only for $\mcC_\eta = \mcO_\eta$, as the proof for $\mcC_\eta = \mcO^*_\eta$ is entirely analogous, with $\ob^*$ in place of $\ob$ and $\Ob^*$ in place of $\Ob$.\\

Clearly (iii) implies both (ii) and (iv).  It is clear from the lemma that (ii) implies (iii), and that (ii) and (iv) are equivalent.  These implications hold for any specified $\eta$.  So it remains to show that (i) and (ii) are equivalent.\\

Suppose (i) holds.  Then $G$ has finitely many normal subgroups of index $n$ for any integer $n$, so $\Ilhd_n(G)$ has finite index.  It follows by Corollary \ref{hjicor} that $\Ob_G(\Ilhd_n(G))$ also has finite index, so $\ob_G(n)$ is finite.  This implies (ii) by taking $\eta = \ob_G$.\\

Suppose (i) is false.  Then by Corollary \ref{hjicor}, there is an open subgroup $H$ of $G$ such that $\Ob_G(H)$ has infinite index in $G$.  Now $H$ has index $h$ say, so that $\Ilhd_h(G) \leq H$.  It follows that $\OI_h(G)$ must be contained in $\Ob_G(H)$, and so $\ob_G(h) = |G:\OI_h(G)| = \infty$.  This implies that (ii) is also false.\end{proof}

As an illustration, consider a pro-$p$ group $G$ of finite obliquity $o$.  As mentioned in \cite{BGJMS}, this also implies that there is some constant $w$ such that $|\gamma_i(G):\gamma_{i+1}(G)| \leq w$.  It is proved in \cite{BGJMS} that the condition of finite obliquity is equivalent to the following:\\

There exists a constant $c$ such that for every normal subgroup $N$ of $G$, and for every normal subgroup $M$ not contained in $N$, we have $|N:N \cap M| \leq p^c$.\\

Lower and upper bounds for $\ob_G$ can easily be derived in terms of these invariants, from which follows a characterisation of the pro-$p$ groups $G$ for which $\ob_G$ is bounded by a linear function.

\begin{prop}(i) The $\ob$-function of $\bZ_p$ is given by $\ob_{\bZ_p}(n) = p^k$, where $k$ is the largest integer such that $p^k \leq n$.  In particular $\ob_{\bZ_p}(n) \leq n$ for all $n$.\\

(ii) Let $G$ be a pro-$p$ group of finite obliquity, with invariants as described above.  Then 
\[ \ob_G(p^n) \leq p^{n+c+w+o-2} \]
for all $n$.  In particular, there is a constant $k$ such that $\ob_G(n) \leq kn$ for all $n$.\\

(iii) Let $G$ be a pro-$p$ group for which there is a constant $k$ such that $\ob_G(n) \leq kn$ for all $n$.  Then either $G \cong \bZ_p$, or $G$ has obliquity at most $\log_p(k)$.\end{prop}

\begin{proof}(i) This is immediate from the definitions.\\

(ii) Let $N$ be a normal subgroup of $G$ of index at most $p^n$.  Then $N$ is not properly contained in any lower central subgroup that has index at least that of $N$; the first such, say $\gamma_r(G)$, has index at most $p^{n+w-1}$.  Hence $\Ilhd_{p^n}(G)$ contains $\Ob_G(\gamma_r(G))$, which has index at most $p^{n+w+o-1}$.\\

Now let $M$ be a normal subgroup of $G$ not contained in $\Ilhd_{p^n}(G)$.  Then $M$ is not contained in some normal subgroup $K$ of index at most $p^n$.  Hence $M$ properly contains a normal subgroup $M \cap K$ of $G$ of index at most $p^{n+c}$.  In particular, $M$ is of index at most $p^{n+c-1}$, so contains $\Ilhd_{p^{n+c-1}}(G)$.  Thus $\OI_{p^n}(G)$ contains $\Ilhd_{p^{n+c-1}}(G)$, a subgroup of index at most $p^{n+c+w+o-2}$.\\

(iii) By Theorem B, $G$ is finite or just infinite.  We may assume $G$ is not $\bZ_p$, which ensures that all lower central subgroups are open (see \cite{BGJMS}).  Let $H$ be a lower central subgroup, of index $h$ say.  Then $H$ contains $\Ilhd_{h}(G)$, so $\Ob_G(H)$ contains $\OI_{h}(G)$, which in turn is a subgroup of $G$ of index at most $kh$.  Hence $|H:\Ob_G(H)|$ is at most $k$.\end{proof}

We now consider profinite branch groups.  The definitions given here are mostly based on those of Grigorchuk in \cite{GriNH}, which should be consulted for a more detailed account, and for constructions of such groups (including the group now generally known as the profinite Grigorchuk group).

\paragraph{Definitions}A \emph{rooted tree} $T$ is a tree with a distinguished vertex, labelled $\emptyset$.  We require each vertex to have finite degree, though the tree itself will be infinite in general.  The \emph{norm} $|u|$ of a vertex $u$ is the distance from $\emptyset$ to $u$; the \emph{$n$-th layer} is the set of vertices of norm $n$.  Denote by $T_{[n]}$ the subtree of $T$ induced by the vertices of norm at most $n$; by our assumptions, $T_{[n]}$ is finite for every $n$.  Write $\Aut(T)$ for the (abstract) group of graph automorphisms of $T$ that fix $\emptyset$.  Then $\Aut(T)$ also preserves the norm, and so there are natural homomorphisms from $\Aut(T)$ to $\Aut(T_{[n]})$, with kernel denoted $\St_{\Aut(T)}(n)$, the \emph{$n$-th level stabiliser}.  Declare the level stabilisers to be open; this generates a topology on $\Aut(T)$, turning $\Aut(T)$ into a profinite group.

\paragraph{Definitions}Let $G$ be a subgroup of $\Aut(T)$.  Then $G$ is said to act \emph{spherically transitively} if it acts transitively on each layer.  Given a vertex $v$, write $T_v$ for the rooted tree with root $v$ induced by the vertices descending from $v$ in $T$.  Define $U^G_v$ to be the group of automorphisms of $T_v$ induced by the stabiliser of $v$ in $G$, and define $L^G_v$ to be the subgroup of $G$ that fixes $v$ and every vertex of $T$ outside $T_v$.  Note that if $G$ acts spherically transitively, the isomorphism types of $U^G_v$ and $L^G_v$ depend only on the norm of $v$; also, there are natural embeddings
\[ L^G_{v_1} \times \dots \times L^G_{v_k} \leq \St_G(n) \leq U^G_{[n]} := U^G_{v_1} \times \dots \times U^G_{v_k},\]
where $v_1, \dots, v_k$ are all the vertices at a given level.  Now $G$ is a \emph{branch group} if $G$ acts spherically transitively and $|U^G_{[n]} : L^G_{v_1} \times \dots \times L^G_{v_k}|$ is finite for all $n$.  Say $G$ is \emph{self-reproducing at $v$} if there is an isomorphism from $T$ to $T_v$ that induces an isomorphism from $G$ to $U^G_v$.  (The definition of \emph{self-reproducing} given in \cite{GriNH} is that this should hold at every vertex.)

\begin{prop}Let $G$ be a just infinite profinite branch group acting on the rooted tree $T$, such that $G$ is self-reproducing at some vertex $v$.  Then there is a constant $c$ such that $\ob_G(n) \leq c^n$ for all $n$.\end{prop}

\begin{proof}Since $G$ is a just infinite profinite group, and the subgroups $\St_G(n)$ are all open in $G$, we can define a function $f$ from $\bN$ to $\bN$ by the property that $\Ob_G(\St_G(n))$ contains $\St_G(n+f(n))$ but not $\St_G(n+f(n)-1)$, for all $n$.  Suppose $|v|=k$, and consider a normal subgroup $K$ of $G$ not contained in $\St_G(k+n)$.  If $K$ is not contained in $\St_G(k)$, then it contains $\St_G(k+f(k))$.  Otherwise, there is some vertex $u$ of norm $k$ such that $K$ acts non-trivially on $(T_u)_{n}$; since $G$ is spherically transitive, we may take $u=v$.  This means that $K/\St_K(T_v)$ contains $\Ob_V(\St_V(n))$, where $V = U^G_v$; since $V \cong G$ as groups of tree automorphisms, we have in turn $\Ob_V(\St_V(n)) \geq \St_V(n+f(n))$.  Since $K$ is normal in $G$, it follows that $K$ induces all automorphisms of $T$ occurring in $G$ that fix the layers up to $k+n+f(n)$, and hence $K$ contains $\St_G(k+n+f(n))$.  Thus $\Ob_G(\St_G(k+n))$ contains $\St_G(k+f(k)) \cap \St_G(k+n+f(n))$, which means that $f(k+n) \leq \max \{f(n),f(k)\}$.  By induction on $n$, this implies $f(n) \leq r$ for all $n$, where $r = \max_{1 \leq i \leq k} f(i)$.\\

Let $N$ be a normal subgroup of index at most $n$, where $n \geq 2$.  Let $l(n)$ be the greatest integer such that $\St_G(l(n))$ has index less than $n$.  Then $N$ is not properly contained in $\St_G(l(n)+1)$, so it contains $\St_G(l(n)+1+f(l(n)+1))$, and hence $\Ob_G(N)$ contains $\St_G(l(n)+1+2f(l(n)+1))$, which in particular contains $\St_G(l(n)+2r+1)$.  Hence $\OI_n(G)$ contains $\St_G(l(n)+2r+1)$.  This means that
\begin{align*}
\ob_G(n) &\leq |G:\St_G(l(n))||\St_G(l(n)):\St_G(l(n)+2r+1)| \\
&\leq n|\St_G(l(n)):\St_G(l(n)+2r+1)|.
\end{align*}
By applying the self-reproducing property of $G$ repeatedly, we obtain an embedding 
\[ \frac{\St_G(l(n))}{\St_G(l(n)+2r+1)} \hookrightarrow \frac{\St_G(t)}{\St_G(t+2r+1)} \times \dots \times \frac{\St_G(t)}{\St_G(t+2r+1)},\]
where $t$ is the integer in the interval $(0,k]$ such that $l(n) \equiv t$ modulo $k$, and the direct factors on the right correspond to the vertices of $T$ of norm $l(n)$ descending from a given vertex of norm $t$.  Since $G$ is spherically transitive, there are less than $n$ vertices of $T$ of norm $l(n)$, so that
\[ \ob_G(n) \leq n(\max_{0 < t \leq k} |\St_G(t):\St_G(t+2r+1)|)^n \]
from which the result follows.\end{proof}

We conclude this section by proving Theorem C.

\begin{lem}Let $G$ be a profinite group, and let $H$ be a subgroup of $G$ of index $h$.  Then:

(i) $\Ilhd_n(G) \geq \Ilhd_n(H) \geq \Ilhd_{tn^h}(G)$, where $t=|G:\Core_G(H)|$;

(ii) If $G$ is just infinite, then $\Ilhd_{hn}(G) \geq \Ilhd_n(H)$ for sufficiently large $n$.\end{lem}

\begin{proof}
(i) If $K$ is a normal subgroup of $G$ of index at most $n$, then $H \cap K$ is a normal subgroup of $H$ of index at most $n$.  On the other hand, let $L$ be a normal subgroup of $H$ of index at most $n$.  Then $M = L \cap \Core_G(H)$ has index at most $n$ in $\Core_G(H)$, and $M$ has at most $h$ conjugates in $G$, all of which are contained in $\Core_G(H)$, so that $\Core_G(M)$ has index at most $n^h$ in $\Core_G(H)$, and hence index at most $tn^h$ in $G$.  Thus every normal subgroup of $H$ of index at most $n$ contains a normal subgroup of $G$ of index at most $tn^h$.\\

(ii) If $G$ is just infinite, there is some integer $m_1$ such that $H$ contains every normal subgroup of $G$ of index at least $hm_1$; furthermore, there is some $m_2$ such that any normal subgroup of $G$ of index less than $hm_1$ contains a normal subgroup of $G$ of index at least $hm_1$, but at most $hm_2$.  The claimed equality holds for any $n \geq m_2$.\end{proof}

\begin{proof}[Proof of Theorem C]For part (i), we may assume that $n$ is large enough that $\Ob_G(H) \geq \Ilhd_{hn}(G) \geq \Ilhd_n(H)$.  The claimed inequalities are demonstrated by the relationships between subgroups given below:
\begin{align*}\OI_{hn}(G) &= \Ilhd_{hn}(G) \cap \bigcap \{ N \unlhd_o G \mid N \not\le \Ilhd_{hn}(G) \} \\
&\geq \Ilhd_{hn}(G) \cap \bigcap \{ N \unlhd_o H \mid N \not\le \Ilhd_{hn}(G) \} \cap \Ob_G(H) \\
&\geq \Ilhd_n(H) \cap \bigcap \{ N \unlhd_o H \mid N \not\le \Ilhd_n(H) \} \\
&= \OI_n(H).\\
\OI^*_n(H) &= \Ilhd_n(H) \cap \bigcap \{ L \leq_o H \mid \Ilhd_n(H) \leq N_H(L), \; L \not\le \Ilhd_n(H) \} \\
&\geq \Ilhd_{tn^h}(G) \cap \bigcap \{ L \leq_o G \mid \Ilhd_{tn^h}(G) \leq N_G(L), \; L \not\le \Ilhd_{tn^h}(G) \} \\
&= \OI^*_{tn^h}(G).\\
\OI^*_n(G) &= \Ilhd_n(G) \cap \bigcap \{ H \leq_o G \mid \Ilhd_n(G) \leq N_G(H), \; H \not\le \Ilhd_n(G) \} \\
&\geq \bigcap_{L \in \mcI_n} \Ilhd_n(L) \cap \bigcap_{L \in \mcI_n} \bigcap \{ H \unlhd_o L \mid H \not\le \Ilhd_n(L) \}\\
&= \bigcap_{L \in \mcI_n} \OI^*_n(L).\end{align*}\end{proof}

\section{Isomorphism types of normal sections and open subgroups}

\begin{prop}Let $G$ be a just infinite profinite group.  Let $M$ and $N$ be open normal subgroups of $G$ such that $N \leq M$, and let $H$ be an open subgroup of $G$, with $\Core_G(H)$ of index $h$.  Then at least one of the following holds:\\

(i) $M/N$ is abelian;\\

(ii) $H$ contains both $M$ and the centraliser of $M/N$;\\

(iii) $M$ contains the open subgroup $\Ob_G(\Ob_G(H))$, and so $|G:M| \leq \ob_G(\ob_G(h))$.
\end{prop}
 
\begin{proof}Assume (i) and (ii) are false.  Since $M$ is a normal subgroup of $G$, to demonstrate (iii) it suffices to prove that $M$ is not properly contained in $\Ob_G(H)$.  Let $K$ be the centraliser of $M/N$ in $G$; note that since (i) is false, $K$ does not contain $M$.  If $H$ does not contain $M$, then $M$ contains $\Ob_G(H)$, so we may assume $H$ contains $M$.  It now follows that $H$ does not contain $K$, by the assumption that (ii) is false.  Since $K$ is normal in $G$, it must contain $\Ob_G(H)$, and hence $\Ob_G(H)$ cannot contain $M$.\end{proof}

\begin{proof}[Proof of Theorem D](i) We may assume that $G$ is not a residually-$\mcA$ group, so $O^\mcA(G)$ has finite index.  Let $M$ and $N$ be normal subgroups such that $M/N$ is a $\mcF$-group, and let $H=C_G(M/N)$.  Then $G/H$ is a $\mcA$-group, and hence an image of $G/O^\mcA(G)$.  On the other hand, $H$ does not contain $M$.  By the above proposition, this means that $|G:M|$ is bounded by a function of $G$ and $|G/O^\mcA(G)|$, and hence there are only finitely many possibilities for $M$.  This means that for some open normal subgroup $M$, there must be infinitely many images of $M$ that are $\mcF$-groups.  Hence $O^\mcF(M)$ is a normal subgroup of $G$ of infinite index, and hence trivial, so that $M$ is residually-$\mcF$.\\

(ii) Since $G$ is not virtually abelian, $H'$ is an open normal subgroup of $G$.  Since $G$ is a hereditary $\Phi^\lhd$-group, it follows that $\Phi^\lhd(H')$ is a proper normal subgroup of $H'$ of finite index, so $H/\Phi^\lhd(H')$ is finite and non-abelian.  The result follows by part (i) applied to $\mcF = [H/\Phi^\lhd(H')]$.\end{proof}

\paragraph{Remark} Part (i) of the above proposition does not extend to abelian sections.  Indeed, given any positive integer $n$ and a prime $p$, then any just infinite pro-$p$ group has infinitely many abelian normal sections of order $p^n$; this is clear for $\bZ_p$, and for any non-nilpotent pro-$p$ group $G$ one can take suitable sections inside $\gamma_k(G)/\gamma_{2k}(G)$ for any $k \geq n$.

\begin{cor}\label{vaiso}Let $G$ be a just infinite profinite group.  Then $G$ has infinitely many isomorphism types of open normal subgroup if and only if $G$ is not virtually abelian.\end{cor}

\begin{proof}Suppose $G$ is virtually abelian.  Then $G$ has an open normal subgroup $V$ which is a free abelian pro-$p$ group.  Every open subgroup of $G$ contained in $V$ is isomorphic to $V$, and by Theorem A, all but finitely many open normal subgroups of $G$ are contained in $V$.  Hence $G$ has only finitely many isomorphism types of open normal subgroup.  The converse follows from part (ii) of Theorem D.\end{proof}

Recall the definition of index-stability given in the introduction.  This property, or the absence of it, is important for understanding the commensurators of just infinite profinite groups, in the sense of \cite{BEW}.\\

A \emph{virtual automorphism} of a profinite group $G$ is an isomorphism between two open subgroups.  Two virtual automorphisms are regarded as equivalent if they agree on an open subgroup of $G$.  Composition of virtual automorphisms is defined up to equivalence by composing suitable equivalence class representatives.  Under this composition, the set of equivalence classes of virtual automorphisms of $G$ forms an abstract group, the \emph{commensurator} $\Comm(G)$ of $G$.  The commensurator depends only on the commensurability class of $G$.\\

Let $H$ and $K$ be isomorphic open subgroups of $G$.  Given an isomorphism $\theta$ from $H$ to $K$, write $\iota(\theta)$ for $|G:H|/|G:K|$.  This is clearly invariant under equivalence, so $\iota(\phi)$ is defined for $\phi \in \Comm(G)$ as $\iota(\theta)$ for any $\theta$ representing $\phi$.  This defines a function $\iota$ from $\Comm(G)$ to the multiplicative group $\bQ^\times_+$ of positive rationals, which we call the \emph{index ratio}.  Let $VZ(G)$ be the union of the finite conjugacy classes of $G$.  In \cite{BEW}, a topology (called the strong topology) is defined on $\Comm(G)$ so that it becomes a topological group $\Comm(G)_S$, and it is shown that if $VZ(G)=1$, then $\Comm(G)_S$ is locally compact, and the index ratio is in fact the modular function for this topology.  For just infinite profinite groups, it can easily be seen that $VZ(G)=1$ if and only if $G$ is not virtually abelian.\\

So Question \ref{indexq} from the introduction is equivalent to the following question:\\

Let $G$ be a (hereditarily) just infinite profinite group which is not virtually abelian.  Is the index ratio of $\Comm(G)$ (or equivalently the modular function of $\Comm(G)_S$) necessarily trivial?\\

This appears to be a difficult question in general, in either form.  First, some general comments about the index ratio.

\begin{lem}\label{irhom}Let $G$ be a profinite group.  Then the index ratio $\iota$ is a homomorphism of abstract groups from $\Comm(G)$ to $\bQ^\times_+$.  In particular, if $G$ is index-unstable, then for all $k \in \bN$ there exists $\phi \in \Comm(G)$ such that $\iota(\phi) > k$.\end{lem}

\begin{proof}Let $\phi,\psi \in \Comm(G)$, and let $\phi'$ and $\psi'$ be representatives of $\phi$ and $\psi$ respectively such that the composition $\phi' \psi'$ is defined.  Let $H$ be the domain of $\phi'$.  Then
\[\iota(\phi\psi)= \frac{|G:H|}{|G:H^{\phi'\psi'}|}=\frac{|G:H|}{|G:H^{\phi'}|}\frac{|G:H^{\phi'}|}{|G:H^{\phi'\psi'}|} = \iota(\phi)\iota(\psi).\]
The conclusions are now clear.\end{proof}

Write $H \unlhd^2_o G$ to indicate that $H$ is an open subnormal subgroup of $G$ of defect at most $2$.

\begin{lem}\label{subnorlem}Let $G$ be a profinite group.  Let $H$ and $K$ be open subgroups of $G$, and suppose $\theta$ is an isomorphism from $H$ to $K$.  Then there are subgroups $H^* \leq H$ and $K^* \leq K$, with $H^* \unlhd^2_o G$ and $K^* \unlhd_o G$, such that the restriction of $\theta$ to $H^*$ induces an isomorphism from $H^*$ to $K^*$.\end{lem}

\begin{proof}Let $H_2$ be the core of $H$ in $G$, and let $K_2$ be its image under $\theta$.  Now let $K^*$ be the core of $K_2$ in $G$, and let $H^*$ be its preimage under $\theta$.  By construction, $K^*$ is normal in $G$, and hence normal in $K_2$.  Since $\theta$ maps $H_2$ isomorphically to $K_2$, this means that $H^*$ must be the corresponding normal subgroup of $H_2$.  But $H_2$ is normal in $G$, so $H^* \unlhd^2_o G$.\end{proof}

\paragraph{Definition}Let $G$ be a just infinite profinite group that is not virtually abelian, and let $N$ be an open normal subgroup of $G$.  We define the following invariant of $G$:
\[ j_N(G) = \frac{\inf \{|G:M| \mid M \cong N, \; M \unlhd_o G\}}{\inf \{|G:M| \mid M \cong N, \; M \unlhd^2_o G\}}.\]\\

Clearly, if $G$ is index-stable then $j_N(G)=1$ for all $N \unlhd_o G$.  In fact, there is a strong converse to this statement.

\begin{prop}\label{jbdprop}Let $G$ be a just infinite profinite group.  Suppose that there are infinitely many isomorphism types of open normal subgroup $N$ of $G$ for which $j_N(G) \leq k$, for some constant $k$.  Then $G$ is index-stable.\end{prop}

\begin{proof}By Corollary \ref{vaiso}, $G$ is not virtually abelian.  Suppose $G$ is index-unstable.  Then by Lemmas \ref{irhom} and \ref{subnorlem}, there are isomorphic subgroups $H$ and $K$ of $G$ such that $|G:H|/|G:K| > k$, and such that $H \unlhd^2_o G$ and $K \unlhd_o G$.  Now $H$ contains all but finitely many normal subgroups of $G$, so all but finitely many isomorphism types of open normal subgroups of $G$ occur only as subgroups of $H$.  This means that there is a normal subgroup $N$ of $G$ such that $j_N(G) \leq k$, and such that all normal subgroups of $G$ isomorphic to $N$ are subgroups of $H$; take $N$ to be of least possible index.  Then $N^\theta \unlhd^2_o G$, where $\theta$ is any isomorphism from $H$ to $K$, and $N^\theta$ is isomorphic to $N$.  Since $N$ was chosen to be of least possible index, it follows that $j_N(G) \geq \iota(\theta) > k$, a contradiction.\end{proof}

We now have enough information to prove Theorem E.

\begin{proof}[Proof of Theorem E]
(i) Suppose $G$ is not virtually abelian.  Then there is an open subgroup $N$ of $H$ such that $N$ is normal in $G$ but $N^\theta$ is not normal in $G$, as shown in the proof of Proposition \ref{jbdprop}.  But then $N$ is normal in $H$, and so $N^\theta \lhd H^\theta = G$, a contradiction.\\

(ii) Let $N=O_\mcX(G)$ for some class of groups $\mcX$, and suppose $N$ is non-trivial.  Let $M$ be a subnormal subgroup of $G$ isomorphic to $N$.  Then by definition, $M$ is generated by its subnormal $\mcX$-subgroups.  But these are then subnormal in $G$, and so contained in $N$.  Hence $M \leq N$, demonstrating that $j_N(G)=1$.  Hence the non-trivial radicals of $G$ form an infinite set of pairwise non-isomorphic open normal subgroups $N$ satisfying $j_N(G)=1$.  By Proposition \ref{jbdprop}, this ensures that $G$ is index-stable.\end{proof}

\section{Acknowledgments}This paper is based on results obtained by the author while under the supervision of Robert Wilson at Queen Mary, University of London (QMUL).  The author would also like to thank Charles Leedham-Green for introducing the author to the study of profinite groups, for his continuing advice and guidance in general and for his detailed feedback and corrections concerning this paper; Yiftach Barnea for his advice on the presentation of some of the proofs and for informing the author about the recent papers \cite{BGJMS} and \cite{BEW}; and Pavel Zalesskii for notifying the author of \cite{Zal}.  The author acknowledges financial support provided by EPSRC and QMUL for the duration of his doctoral studies.

\end{document}